\documentclass[11pt]{article}

\usepackage{amsfonts}
\usepackage{amssymb, amsmath}
\usepackage{natbib}
\usepackage{lscape}
\usepackage{dsfont}
\usepackage{mathrsfs}
\usepackage{bm}
\usepackage[pdftex,plainpages=false,colorlinks,hyperindex,bookmarksopen,linkcolor=red,citecolor=blue,urlcolor=blue]{hyperref}

\DeclareMathAlphabet{\mathpzc}{OT1}{pzc}{m}{it}
\bibpunct{[}{]}{;}{n}{,}{,}
\newtheorem{te}{Theorem}[section]

\newtheorem{os}[te]{Remark}
\newtheorem{prop}[te]{Proposition}

\newenvironment{proof}[1][Proof]{\noindent\textbf{#1.} }{\ \rule{0.5em}{0.5em}}
\numberwithin{equation}{section}
\allowdisplaybreaks

\begin{document}
\title{Time-inhomogeneous\\
fractional Poisson processes defined by the multistable subordinator}

\author{Luisa Beghin \thanks{%
luisa.beghin@uniroma1.it,
Department of Statistical Sciences, Sapienza - University of Rome
} \and Costantino Ricciuti \thanks{costantino.ricciuti@uniroma1.it (corresponding author),
Department of Statistical Sciences, Sapienza-
 University of Rome. }}
\maketitle

\begin{abstract}
The space-fractional and the time-fractional Poisson processes are two
well-known models of fractional evolution. They can be constructed as
standard Poisson processes with the time variable replaced
by a stable subordinator and its inverse, respectively. The aim of this
paper is to study non-homogeneous versions of such models, which can be defined
by means of the so-called multistable subordinator
(a jump process with non-stationary increments), denoted by $H:=H(t),t\geq0$. Firstly,
we consider the Poisson process time-changed by $H$ and we obtain its explicit distribution
and governing equation. Then, by using the right-continuous inverse of $H$,
we define an inhomogeneous analogue of the time-fractional Poisson process.


\end{abstract}

\section{Introduction}
Non-homogeneous subordinators are univariate additive processes with non-decreasing
sample paths. Such processes, together with their right
continuous inverses, have recently been studied in \cite{Orsingher}, where
they are also used as random clock for time-changed processes. Recall that
an additive process is characterized by independent increments and is
stochastically continuous, null at the origin and with cadlag trajectories
(for a deeper insight consult \cite{Sato}). If, in addition, we assume
stationarity of the increments, additive processes reduce to the standard L\'{e}%
vy ones.

A non-homogeneous subordinator (without drift) is completely characterized
by a set $\{\nu _{t},t\geq 0\}$ of L\'{e}vy measures on $\mathbb{R}^{+}$, such
that
\begin{equation*}
\nu _{t}(0)=0\qquad \int_{0}^{\infty }(x\wedge 1)\nu _{t}(dx)<\infty, \qquad
t\geq 0.
\end{equation*}%
If $\nu _{t}(\mathbb{R}^{+})<\infty $, for any $t\geq 0$, then the process
reduces to an inhomogeneous Compound Poisson Process (hereafter CPP), while
condition $\nu _{t}(\mathbb{R}^{+})=\infty $, for any $t \geq 0$, ensures that the
process is strictly increasing almost surely. Under suitable conditions (see
\cite{Orsingher}), the Laplace transform of the increments of an inhomogeneous
subordinator $T$ has the form
\begin{equation} \label{cosa 1}
\mathbb{E}e^{-u(T(t)-T(s))}=e^{-\int_{s}^{t}f(u,\tau )d\tau }\qquad 0\leq
s\leq t,
\end{equation}%
where $u\rightarrow f(u,t)$ is a Bernstein function for each $t \geq 0$, having
the following form
\begin{equation} \label{cosa 2}
f(u,t)=\int_{0}^{\infty }(1-e^{-ux})\nu _{t}(dx).
\end{equation}

Among inhomogeneous subordinators, we are particularly interested in the
so-called multistable subordinators (see \cite{molcha}, \cite{Orsingher},
\cite{legueve2}). These processes extend the well-known stable subordinators
by letting the stability index $\alpha $ evolve autonomously in time:
for this reason they have been proved to be particularly useful in modelling 
phenomena, both in finance and in natural sciences, where the intensity of the jumps is 
itself time-dependent. The multistable subordinator is
fully characterized by a L\'{e}vy measure of the form
\begin{equation*}
\nu _{t}(dx)=\frac{\alpha (t)x^{-\alpha (t)-1}}{\Gamma (1-\alpha (t))}%
dx,\qquad x>0,
\end{equation*}%
where $t\rightarrow \alpha (t)$ has values in $(0,1)$. Throughout the paper
we will denote a multistable subordinator by $H:=\{H(t),t\geq 0\}$. It is
known (see \cite{Orsingher}) that, for each $t \geq 0$, the random variable $H(t)$
is absolutely continuous and its density solves
\begin{equation*}
\frac{\partial }{\partial t}q(x,t)=-\frac{\partial ^{\alpha (t)}}{\partial
x^{\alpha (t)}}q(x,t),\qquad q(x,0)=\delta (x),
\end{equation*}%
where $\frac{\partial ^{\alpha (t)}}{\partial x^{\alpha (t)}}$ is the
Riemann-Liouville derivative with time-varying order.

Since, in this case, $f(u,t)=u^{\alpha (t)}$, the increment from $s$ to $t$
has Laplace transform
\begin{equation*}
\mathbb{E}e^{-u(H(t)-H(s))}=e^{-\int_{s}^{t}u^{\alpha (\tau )}d\tau },\qquad
0\leq s\leq t.
\end{equation*}%
The first part of the present paper has been inspired by \cite{Polito}, \cite{Polito2} and \cite%
{Toaldo}. In particular, in \cite{Polito} the authors study the composition of a Poisson
process with a stable subordinator. The resulting process, called
space-fractional Poisson process, is also a subordinator, namely a point
process with upward jumps, with arbitrary, integer size.\\
Let now $%
N:=\{N(t),t \geq 0\}$ be a homogeneous Poisson process with intensity $\lambda >0$
and let $H$ be a multistable subordinator independent of $N$. We
consider here the point process $X:=\{X(t), t\geq 0\}$, where, for any $t\geq 0$, $X(t):= N(H(t))$,
 with positive integer values, that
we call Space-Multifractional Poisson Process (hereafter SMPP). We prove
that its state probabilities $p_{k}(t)=\Pr \{X(t)=k\}$ satisfy the following
system of difference-differential equations:
\begin{equation}
\begin{cases}
\frac{d}{dt}p_{k}(t)=-\lambda ^{\alpha (t)}(I-B)^{\alpha (t)}p_{k}(t),\qquad
k=0,1,2... \\
p_{k}(0)=\delta _{k,0},\label{eqSMPP}%
\end{cases}%
\end{equation}%
where $B$ is the shift operator such that $Bp_{k}(t)=p_{k-1}(t)$, and $\delta _{k,0}$
denotes the Kronecker delta function.
The first equation in (\ref{eqSMPP}) is a time-inhomogeneous extension of the Space-Fractional Poisson
governing equation studied in \cite{Polito} (see also \cite{Beghin3} for the compound case).
 This result confirms the
validity of the time-inhomogeneous version of the Phillips' formula, which
was proved in \cite{Orsingher} for
self-adjoint Markov generators only, and therefore it could not be taken for
granted in the case of Poisson generators. In other words, referring to the
general theory of Markov processes (see, for example, \cite{Kolokoltsov}), we say that the evolution of $%
X$ is governed by a propagator (or two parameter semigroup) with
time-dependent adjoint generator given by $\lambda ^{\alpha
(t)}(I-B)^{\alpha (t)}$.

In the second part of the present paper we study the so called Time-Multifractional Poisson
Process (hereafter TMPP). It is obtained by time-changing the standard
Poisson process via the right continuous inverse of a multistable
subordinator, which is defined as
\begin{equation*}
L(x)=inf\{t \geq 0:H(t)>x\}.
\end{equation*}%
We recall that the classical time-fractional Poisson process is a renewal
process with i.i.d Mittag-Leffler waiting times, having a deep
connection to fractional calculus. It has been introduced and studied by
\cite{Mainardi}, \cite{Beghin2}, \cite{Beghin}, \cite{Politi}, \cite{Gorenflo}, \cite{Garra} and many others.
In \cite{Meerschaert1} the authors show that it can be constructed by time-changing
a Poisson process via an independent inverse stable subordinator.\\
The idea of time-changing Markov processes via non-homogeneous subordinators has been
developed in \cite{Orsingher}. Moreover, in \cite{molcha} the TMPP arises as
a scaling limit of a continuous time random walk, but its distributional
properties are not investigated there. We prove here that non-homogeneity has an
impact on the distribution of the waiting times, which are independent but no longer
identically distributed.

Very recently, some authors (\cite{Leonenko} and \cite{Vellaisamy}) considered some extensions of the time-fractional Poisson process,
which are inhomogeneous in a different sense from ours. The difference consists in the fact that they analyse the time-change of an inhomogeneous Poisson process by the inverse of a homogeneous stable subordinator.

\vspace{0.4cm}

\section{Preliminary results}
In view of what follows, we preliminarily need the following extension of Theorem 30.1,
p.197 in \citep{Sato}, to the case of non-homogeneous subordinators.

\begin{prop}
\label{proposizione iniziale} Let $M:=\{M(t),t>0\}$ be a L\'{e}vy subordinator such that $\mathbb{E}e^{-uM(t)}= e^{-tg(u)}$
and let $T:=\{T(t),t>0\}$ be a non-homogeneous subordinator (without drift) with L\'{e}vy
measure $\nu _t$ and Bernstein function $f(\cdot,\cdot)$ as defined in (\ref{cosa 1}) and (\ref{cosa 2}). Let $Z:= \{Z(t)=M(T(t)), t>0  \}$ be the time changed process. Then

\textit{i)} $Z$ is a non-homogeneous subordinator (without drift)

\textit{ii)} $Z$ has time-dependent L\'{e}vy measure
\begin{equation}
\nu _{t}^{\ast }(dx)=\int_{0}^{\infty }\Pr (M(s)\in dx)\nu _{t}(ds)\label{misura di Levy
subordinato}
\end{equation}
\end{prop}

\begin{proof}
\textit{i)}  The fact that $Z$ is non-decreasing is obvious, since $Z$ is given by the composition of non decreasing processes. It remains to prove independence of increments and stochastic continuity.
First we prove that $Z$ has independent increments. By Kac's theorem on characteristic functions (see \cite{applebaum}, p.18), it is sufficient to prove that, for any $0\leq t_1\leq t_2 \leq t_3$,
\begin{align*}
\mathbb{E}e^{i y_1 (Z(t_3)-Z(t_2))+i y_2 (Z(t_2)-Z(t_1)) }= \mathbb{E}e^{i y_1 (Z(t_3)-Z(t_2))} \mathbb{E}e^{ i y_2 (Z(t_2)-Z(t_1))}\qquad \forall (y_1,y_2)\in \mathbb{R}^2.
\end{align*}
For the sake of simplicity, we use the notation $T(t_j)=T_j$. A simple conditioning argument yields
\begin{align*}
 \mathbb{E}e^{i y_1 (Z(t_3)-Z(t_2))+i y_2 (Z(t_2)-Z(t_1)) }
&= \mathbb{E} \bigl [ \mathbb{E}\bigl(e^{i y_1 (M(T_3)-M(T_2))+i y_2 (M(T_2)-M(T_1))}|T_1,T_2,T_3 \bigr) \bigr]\\
&= \mathbb{E} \bigl [ \mathbb{E}\bigl(e^{i y_1 (M(T_3)-M(T_2))}|T_2,T_3\bigr) \mathbb{E}\bigr(e^{ i y_2 (M(T_2)-M(T_1)) } |T_1,T_2\bigr)\bigr],
\end{align*}
 where the last step follows by the fact that $M$ has independent increments. Now, since $M$ has stationary increments, we have
\begin{align} \label{formula intermedia}
\mathbb{E} \bigl [ \mathbb{E}\bigl(e^{i y_1 (M(T_3-T_2))}|T_2,T_3\bigr) \mathbb{E}\bigr(e^{ i y_2 (M(T_2-T_1)) } |T_1,T_2\bigr)\bigr]= \mathbb{E}e^{iy_1M(T_3-T_2)} \mathbb{E}e^{i y_2M(T_2-T_1)},
\end{align}
where, in the last equality, we have taken into account that $T$ has independent increments and thus $M(T_3-T_2)$ and $M(T_2-T_1)$ are stochastically independent. By using again the same conditioning argument, it is now immediate to observe that the right hand side of  (\ref{formula intermedia})  can be written as
\begin{align*}
\mathbb{E}e^{iy_1 (M(T_3)-M(T_2))}\mathbb{E}e^{iy_2 (M(T_2)-M(T_1))},
\end{align*}
since $M$ has stationary increments, and this concludes the proof of the independence of increments of $Z$.

 We now recall that a process $Y(t)$ is said to be stochastically continuous at time $t$ if
 $P(|Y(t+h)-Y(t)|>a)\to 0$, as $h\to 0$, for any $a>0$. Then, denoting by $\mu _{t,t+h}$ the law of $T(t+h)-T(t)$ and using the stationarity of the increments of $M$,  we have that
 \begin{align*}
 \Pr \{|Z(t+h)-Z(t)|>a \}&= \Pr \{ |M(T(t+h))-M(T(t))|>a \}\\
 &= \int _0 ^{\infty}\Pr \{ |M(u)|>a   \} \mu _{t,t+h}(du)\\
 &= \int _0^{\delta} \Pr \{ |M(u)|>a   \} \mu _{t,t+h}(du)+ \int _{\delta}^{\infty} \Pr \{ |M(u)|>a   \} \mu _{t,t+h}(du)\\
 & \leq \sup _{u\in (0,\delta)}\Pr \{ |M(u)|>a   \}+ \Pr \{ |T(t+h)-T(t)|>\delta   \},
 \end{align*}
 where $\delta >0$ can be arbitrarily small. Now, by letting $\delta$ and $h$ go to zero, stochastic continuity of $M$ and $T$ produces the desired result.

\textit{ii)}
By using a simple conditioning argument, we have that
\begin{align*}
\mathbb{E}e^{-uM(T(t))}= & \int _0 ^{\infty} \mathbb{E}e^{-uM(s)}\Pr \{ T(t)\in ds  \}\\
&= \int _0 ^{\infty} e^{-sg(u)} \Pr \{ T(t)\in ds  \}\\
&= e^{-\int _0^t f(g(u),\tau)d\tau}.
\end{align*}
Thus the Bernstein function of $M(T(t))$ has the form
\begin{align*}
f(g(u),\tau)& = \int _0^{\infty} (1-e^{-g(u)z})\nu _{\tau}(dz)\\
&= \int_0 ^{\infty} (1-\mathbb{E}e^{-uM(z)}) \nu _ {\tau}(dz)\\
&= \int _0 ^{\infty}\nu _{\tau}(dz) \int _0 ^{\infty} (1-e^{-ux})\Pr \{ M(z)\in dx \}\\
&= \int _0^{\infty} (1-e^{-ux}) \int _0^{\infty} \Pr \{M(z)\in dx  \} \nu _{\tau}(dz)\\
&= \int _0 ^{\infty} (1-e^{-ux}) \nu ^*_{\tau}(dx)
\end{align*}
and the proof is complete.
\end{proof}

\section{Space-Multifractional Poisson process}

Consider a standard Poisson process $N$, with rate $\lambda >0$, and a
multistable subordinator $H$ with index $\alpha (t)$. We define the SMPP as
the time-changed process $\left\{ N(H(t)),t \geq 0\right\} $.  Such a process is completely
characterized by its time-dependent L\'{e}vy measure and by its transition
probabilities, which are given in the following theorem.

\begin{te}
The SMPP $X(t):=N(H(t)),$ for any $t \geq 0,$

\textit{i)} is a non-homogeneous subordinator and has L\'{e}vy measure
\begin{equation}
\nu _{t}^{\ast }(dx)=\lambda ^{\alpha (t)}\sum_{n=1}^{\infty }(-1)^{n+1}%
\binom{\alpha (t)}{n}\delta _{n}(dx),\label{misura di Levy space
fractional}
\end{equation}

ii) has the following transition probabilities%
\begin{align}
\Pr\{ X(\tau +t)=k+n|X(\tau)=k \}= \begin{cases} \sum _{r=1}^{\infty} \frac{(-1)^{n+r}}{r!} \int _{[\tau,\tau +t]^r}
\lambda ^{\beta_r(s)} \binom{\beta_r(s)}{n}ds_1...ds_r, \qquad n \geq 1\label{probabilita di transizione} \\
\\
 e ^{-\int _{\tau} ^{\tau +t} \lambda ^{\alpha (s)}ds} \qquad n=0,
\end{cases}
\end{align}
where
\begin{equation*}
\beta _{r}(s):=\beta _{r}(s_{1},...,s_{r})=\sum_{j=1}^{r}\alpha (s_{j}).
\end{equation*}%
\end{te}

\begin{proof}
\textit{i)} The fact that $X$ is a non-homogeneous subordinator is a consequence of Prop.
\ref{proposizione iniziale}.
Denoting respectively by $\nu _{t}(dx)$ and $\nu _{t}^{\ast
}(dx) $ the L\'{e}vy measures of $H$ and $X$, we apply (\ref{misura di Levy
subordinato}) and obtain
\begin{align*}
\nu _{t}^{\ast }(dx)& =\int_{0}^{\infty }Pr(N(s)\in dx)\nu _{t}(ds) \\
& =\int_{0}^{\infty }\sum_{k=1}^{\infty }e^{-\lambda s}\frac{(\lambda s)^{k}%
}{k!}\delta _{k}(dx)\frac{\alpha (t)s^{-\alpha (t)-1}}{\Gamma (1-\alpha (t))}%
ds \\
& =\sum_{k=1}^{\infty }\frac{\alpha (t)\lambda ^{\alpha (t)}\Gamma (k-\alpha
(t))}{\Gamma (1-\alpha (t))k!}\delta _{k}(dx) \\
& = \sum _{k=1}^{\infty} \frac{\alpha (t)\lambda  ^{\alpha (t)} (k-\alpha (t)-1)(k-\alpha (t)-2)...(-\alpha(t) ) \Gamma (-\alpha (t))}{k! (-\alpha (t))\Gamma (-\alpha (t))} \delta _k (dx)\\
& =\lambda ^{\alpha (t)}\sum_{k=1}^{\infty }(-1)^{k+1}\binom{\alpha (t)}{k}%
\delta _{k}(dx).
\end{align*}

\textit{ii)} The probability generating
function of the increment $X(\tau + t)-X(\tau)$ has the following form
\begin{align}
G(u,\tau, t)& =\mathbb{E}u^{N(H(\tau +t))-N(H(\tau))}\notag \\
& =\mathbb{E} \bigl [\mathbb{E}\bigl(
u^{N(H(\tau +t)-H(\tau))}| H(\tau), H(\tau +t) \bigr) \bigr]\notag \\
&=e^{-%
\int_{\tau}^{\tau +t}\lambda ^{\alpha (s)}(1-u)^{\alpha (s)}ds}\label{funzione generatrice}.
\end{align}%
By a series expansion we have
\begin{align*}
G(u,\tau, t)&= \sum_{r=0}^{\infty }\frac{(-1)^{r}}{r!}\biggl (\int_{\tau }^{\tau +t}\lambda
^{\alpha (s)}(1-u)^{\alpha (s)}ds\biggr )^{r} \\
&= 1+\sum_{r=1}^{\infty }\frac{(-1)^{r}}{r!}\int_{[\tau ,\tau +t]^{r}}\lambda ^{\beta
_{r}(s)}(1-u)^{\beta _{r}(s)}ds_{1}....ds_{r} \\
& = u^0 \biggl [1+\sum _{r=1}^{\infty } \frac{(-1)^r}{r!}\int _{[\tau,\tau+t]^r}\lambda ^{\beta _r(s)}ds_1...ds_r\biggr ]+\\
& +\sum_{n=1}^{\infty }u^{n}\biggl [\sum_{r=1}^{\infty }\frac{(-1)^{n+r}}{r!}%
\int_{[\tau,\tau +t]^{r}}\lambda ^{\beta _{r}(s)}\binom{\beta _{r}(s)}{n}ds_1...ds_r \biggr ].
\end{align*}%
Thus the increments of the SMPP have distribution
\begin{align}  \label{probabilita di stato}
\Pr\{ X(\tau +t)-X(\tau)=n \}= \begin{cases} \sum _{r=1}^{\infty} \frac{(-1)^{n+r}}{r!} \int _{[\tau,\tau +t]^r}
\lambda ^{\beta_r(s)} \binom{\beta_r(s)}{n}ds_1...ds_r, \qquad n \geq 1,\\
\\
 e ^{-\int _{\tau} ^{\tau +t} \lambda ^{\alpha (s)}ds}, \qquad
 n=0,
\end{cases}.
\end{align}
Now we recall that additive processes are space-homogeneous (see \cite{Sato}%
, p.55), namely the transition probabilities are such that
\begin{equation*}
\Pr \{X(t)\in B|X(s)=x\}=\Pr \{X(t)\in B-x|X(s)=0\}=\Pr
\{X(t)-X(s)\in B-x\},
\end{equation*}%
for any $0\leq s\leq t$ and any Borel set $B\subset \mathbb{R}$. Thus the desired result concerning the transition probabilities holds true.
\end{proof}

\begin{os}
By the same conditioning argument used in (\ref{funzione generatrice}), we
find that the Laplace transform of $X(t)$ reads
\begin{equation*}
\mathbb{E}e^{-\eta N(H(t))}=e^{-\int_{0}^{t}\lambda ^{\alpha (\tau
)}(1-e^{-\eta })^{\alpha (\tau )}d\tau }
\end{equation*}%
and then its Bernstein function is given by%
\begin{equation}
f^{\ast }(\eta ,t)=\lambda ^{\alpha (t)}(1-e^{-\eta })^{\alpha (t)}.\label{Bernstein function}
\end{equation}%
We can check that (\ref{Bernstein function}) can also be obtained by applying
the definition involving the L\'{e}vy measure (\ref{misura di Levy space
fractional})
\begin{align*}
f^{\ast }(\eta ,t )& =\int_{0}^{\infty }(1-e^{-\eta x})\nu _{t}^{\ast
}(dx) \\
& =\lambda ^{\alpha (t)}\sum_{k=1}^{\infty }\binom{\alpha (t)}{k}%
(-1)^{k+1}\int_{0}^{\infty }(1-e^{-\eta x})\delta _{k}(dx) \\
& =\lambda ^{\alpha (t)}\sum_{k=1}^{\infty }\binom{\alpha (t)}{k}%
(-1)^{k+1}(1-e^{-\eta k}) \\
& =\lambda ^{\alpha (t)}(1-e^{-\eta })^{\alpha (t)}.
\end{align*}
\end{os}

\begin{os}
In the limiting case where the stability index is constant, namely
$\alpha (s)=\alpha >0$, the multistable subordinator $H$ reduces
to the classical stable subordinator and thus $X$ is the classical
space fractional process studied in \cite{Polito}, which is a
time-homogeneous process. Indeed, it is straightforward to check
that, if $\alpha $ is constant,
\begin{equation*}
\beta _{r}(s)=\sum_{j=1}^{r}\alpha (s_{j})=r\alpha
\end{equation*}%
and, putting $\tau=0$ by time homogeneity, expression (\ref{probabilita di transizione}) reduces to
\begin{align*}
p_{n}(t)& =\sum_{r=0}^{\infty }\frac{(-1)^{r+n}}{r!}\lambda ^{\alpha r}t^{r}%
\binom{\alpha r}{n} \\
& =\sum_{r=0}^{\infty }\frac{(-1)^{r+n}}{r!}\lambda ^{\alpha r}t^{r}\frac{%
\Gamma (\alpha r+1)}{n!\Gamma (\alpha r-n+1)},
\end{align*}%
which is the one-dimensional distribution computed in \cite{Polito}.
\end{os}

\subsection{Governing equation}
The Phillips' theorem states that, if $\left\{ Y(t),t \geq 0\right\} $ is a Markov
process with generator $A$ and $H$ is a subordinator with Bernstein
function $f(\lambda )$, then $\left\{ Y(H(t)),t \geq 0\right\} $ is a Markov
process with generator $-f(-A)$ (for a deeper insight, consult \cite{Sato} and \cite{toaldo}).
This explains why the state probabilities $p_k(t)=P(X(t)=k)$ of the space-fractional Poisson
process studied in \cite{Polito} are governed by the following system of
difference-differential equations
\begin{equation}
\begin{cases}
\label{governing}\frac{d}{dt}p_{k}(t)=-\lambda ^{\alpha }(I-B)^{\alpha
}p_{k}(t) \\
p_{k}(0)=\delta _{k,0},%
\end{cases}%
\end{equation}%
where $B$ is the shift operator such that $Bp_{k}(t)=p_{k-1}(t)$.

In \cite{Orsingher}, the Phillips' theorem has been partially extended to
time-changed processes $Y(H(t))$, where $H$ is a
non-homogeneous subordinator with Bernstein function $f(u,t)$. Indeed, by
means of a functional analysis approach, the authors proved that $\left\{
Y(H(t)),t \geq 0\right\} $ is an additive process with time-dependent generator $%
-f(-A,t)$, at least when $A$ is self-adjoint. Thus, it is not obvious that
this fact also applies to the SMPP, since the generator $A$ of a standard
Poisson process is not self-adjoint. However, the following proposition
confirms the Phillips' type form of the time-dependent generator.

\begin{prop}
The state probabilities of the SMPP solve the following system of
difference-differential equations
\begin{equation}
\begin{cases}
\label{governing}\frac{d}{dt}p_{k}(t)=-\lambda ^{\alpha (t)}(I-B)^{\alpha
(t)}p_{k}(t), \\
p_{k}(0)=\delta _{k,0}.%
\end{cases}%
\end{equation}%

\end{prop}

\begin{proof}
Let us consider the distribution given in (\ref{probabilita di stato}). Each
multiple integral over $[0,t]^{r}$ is of order $t^{r}$, so that, for small
time intervals, the distribution of the increments has the following form
\begin{align}
\Pr \{X(t+dt)-X(t)=n\}=%
\begin{cases}
1-\lambda ^{\alpha (t)}dt+o(dt)\qquad & n=0
\\
(-1)^{n+1}\lambda ^{\alpha (t)}\binom{\alpha (t)}{n}dt+o(dt)\qquad & n\geq 1\label{probabilita di transizione infinitesime}
\end{cases}%
\end{align}%
By using the expansion $(I-B)^{\alpha (t)}=\sum_{n=0}^{\infty }\binom{\alpha
(t)}{n}(-1)^{n}B^{n}$, equation (\ref{governing}) can be written as
\begin{equation*}
p_{k}(t+dt)=p_{k}(t)(1-\lambda ^{\alpha
(t)}dt)+\sum_{n=1}^{k}p_{k-n}(t)\lambda ^{\alpha (t)}(-1)^{n+1}\binom{\alpha
(t)}{n}dt+o(dt)
\end{equation*}%
which is the forward equation of an inhomogeneous Markov process whose
infinitesimal (time-dependent) transition probabilities have just the form (\ref%
{probabilita di transizione infinitesime}) and this concludes the proof.
\end{proof}

\subsection{ Compound Poisson representation and jump times}

The space-fractional Poisson process introduced in \cite{Polito}
is a counting process with upward jumps of arbitrary size. A
fundamental property is that the waiting times between successive
jumps, $J_{n}$, are i.i.d random variables with common
distribution
\begin{equation*}
\Pr \{J_{n}>\tau \}=e^{-\lambda ^{\alpha }\tau },\quad \forall
n\geq 1.
\end{equation*}%
Thus the jump times $T_{n}=J_{1}+...+J_{n}$ follow a gamma distribution:
\begin{equation*}
\Pr \{T_{n}\in dt\}=\frac{1}{\Gamma (n)}\lambda ^{\alpha
n}t^{n-1}e^{-\lambda ^{\alpha }t},\qquad t\geq 0.
\end{equation*}

A difficulty arises in the SMPP case, where the waiting times $J_{n}$
are neither independent nor identically distributed random variables and $%
T_{n}$ cannot be obtained as the convolution of $J_{1},J_{2},...J_{n}$. By
using (\ref{probabilita di transizione}), the waiting time of the first jump
has distribution
\begin{equation*}
\Pr \{J_{1}>t\}=\Pr \{X(t)=0|X(0)=0\}=e^{-\int_{0}^{t}\lambda
^{\alpha (s)}ds},
\end{equation*}%
while the $n^{th}$ waiting time is such that
\begin{equation*}
\Pr \{J_{n}>t|J_{1}+J_{2}+...J_{n-1}=\tau \}=\Pr \{X(\tau +t)-X(\tau
)=0\}=e^{-\int_{\tau }^{\tau +t}\lambda ^{\alpha (s)}ds}
\end{equation*}%
and this shows that the variables $J_{n},n\geq 1$, are
stochastically dependent.

In order to find the distribution of $T_{n}$, it is convenient to
note that the SMPP is an inhomogeneous CPP in the sense of
\cite{Orsingher}. In Section 1 we recalled that a non-homogeneous
subordinator such that $\nu _{t}(\mathbb{R}^{+})<\infty $, for
each $t \geq 0$,
reduces to a inhomogeneous CPP. As shown in \cite%
{Orsingher}, such a process can be constructed as
\begin{equation*}
Y(t)=\sum_{j=1}^{P(t)}Y_{j},
\end{equation*}%
where $P(t)$ is a time-inhomogeneous Poisson process with intensity $g(t)$
and hitting times $T_{j}=\inf \{t \geq 0:P(t)=j\}$, and $Y_{j}$ are positive and
non-stationary jumps, such that
\begin{equation*}
\Pr \{Y_{j}\in dy|T_{j}=t\}=\psi (dy,t).
\end{equation*}%
We recall that the L\'{e}vy measure of such a process has the form
\begin{equation}
\nu _{t}(dy)=g(t)\psi (dy,t),\label{misura di
Levy compound poisson}
\end{equation}%
whence $\nu _{t}(\mathbb{R}^{+})=g(t)<\infty $, see \cite{Orsingher} for
details.

\begin{te}
Let $X(t)=N(H(t))$ be the SMPP and consider the inhomogeneous CPP
\begin{equation*}
Y(t)=\sum_{j=1}^{P(t)}Y_{j},
\end{equation*}
such that $P$ is a inhomogeneous Poisson process with intensity $%
g(t)=\lambda ^{\alpha (t)}$ and the $Y_{j}$ have distribution
\begin{equation*}
\psi (dx,t)=\Pr \{Y_{j}\in dx|T_{j}=t\}=\sum_{n=1}^{\infty }(-1)^{n+1}\binom{%
\alpha (t)}{n}\delta _{n}(dx).
\end{equation*}
Then

i) $X$ and $Y$ are equal in the f.d.d.'s sense.

ii) The epochs $T_{j}$ at which the jumps of $X$ occur have marginal
distributions
\begin{equation}
\Pr \{T_{j}\in dt\}=\frac{(\int_{0}^{t}\lambda ^{\alpha
(s)}ds)^{j-1}e^{-\int_{0}^{t}\lambda ^{\alpha (s)}ds}}{\Gamma
(j)}\lambda ^{\alpha (t)}dt.\label{hitting times}
\end{equation}
\end{te}

\begin{proof}
\textit{i)} $X$ and $Y$ are both inhomogeneous subordinators. Thus
they are equal in the f.d.d.'s sense if and only if their L\'{e}vy
measures coincide. By (\ref{misura di Levy compound poisson}), the
L\'{e}vy measure of $Y$ is $\nu _{t}(dy)=g(t)\psi (dy,t)$ and it
corresponds to (\ref{misura di Levy space fractional}).

\textit{ii)} The process $P$ has distribution
\begin{equation*}
\Pr \{P(t)=k\}=e^{-\int_{0}^{t}\lambda ^{\alpha (\tau )}d\tau }\frac{%
(\int_{0}^{t}\lambda ^{\alpha (\tau )}d\tau )^{k}}{k!} \qquad k\geq 0.
\end{equation*}%
Now, $P$ governs the epochs $T_{n}$, $n\geq 0$, at which the jumps of $Y$
occur, i.e.
\begin{equation*}
T_{n}=inf\{t\geq 0:P(t)=n\}.
\end{equation*}
To our aim, it is convenient to resort to the deterministic time change $%
t\rightarrow t^{\prime }$ given by the transformation
\begin{equation}
t^{\prime }=M(t)=\int_{0}^{t}\lambda ^{\alpha (s)}ds,\label{deterministic
time change}
\end{equation}%
where $M$ is clearly a continuous and monotonic function.
Thus $P(t)$ transforms into
\begin{equation*}
\Pi (t^{\prime })=P(M^{-1}(t^{\prime })).
\end{equation*}%
By virtue of the Mapping Theorem (see \cite{Kingman}, p.18), $\Pi
(t^{\prime })$ is also a Poisson process. Moreover it is
homogeneous with intensity $1$ and its hitting times
$T_{j}^{\prime }$ follow a $Gamma(1,j)$ distribution, i.e.
\begin{equation*}
\Pr \{T_{j}^{\prime }\in dt^{\prime }\}=\frac{(t^{\prime
})^{j-1}e^{-t^{\prime }}}{\Gamma (j)}dt^{\prime }.
\end{equation*}%
The hitting times $T_{j}$ of $P(t)$ are the images of $T_{j}^{\prime }$
under the transformation $M^{-1}$: thus, by a simple transformation
of the probability density of $T_{j}^{\prime }$ , we obtain (\ref{hitting times}).
\end{proof}

\subsection{Upcrossing times}

Let $\mathcal{T}_{k}$ be the time of the first upcrossing of the level $k$,
i.e.
\begin{equation*}
\mathcal{T}_{k}=inf\{t \geq 0:X(t)\geq k\}.
\end{equation*}%
We now find two equivalent expressions for its distribution. The
first one is a generalization of the result given in
\cite{Toaldo}, p.8:
\begin{align}
Pr(\mathcal{T}_{k}>t)& =Pr(X(t)<k)  \notag
\label{Tempi di primo sorpasso prima formula} \\
& =\sum_{n=0}^{k-1}Pr(X(t)=n)  \notag \\
& =\sum_{n=0}^{k-1}\int_{0}^{\infty }Pr(N(s)=n)Pr(H(t)\in ds)  \notag \\
& =\sum_{n=0}^{k-1}\int_{0}^{\infty }e^{-\lambda s}\frac{(\lambda s)^{n}}{n!}%
Pr(H(t)\in ds)  \notag \\
& =\sum_{n=0}^{k-1}\frac{(-\lambda )^{n}}{n!}\frac{d^{n}}{d\lambda ^{n}}%
\int_{0}^{\infty }e^{-\lambda s}Pr(H(t)\in ds)  \notag \\
& =\sum_{n=0}^{k-1}\frac{(-\lambda )^{n}}{n!}\frac{d^{n}}{d\lambda ^{n}}%
e^{-\int_{0}^{t}\lambda ^{\alpha (\tau )}d\tau }.
\end{align}%
The second one allows us to write the survival function of $\mathcal{T}_{k}$ in terms
of the state probability of the level $k$ in the following way:
\begin{equation}
Pr(\mathcal{T}_{k}>t)=1-k\int_{0}^{\lambda }d\lambda ^{\prime }\frac{1}{\lambda
^{\prime }}\,Pr(X_{\lambda ^{\prime }}(t)=k).\label{Tempi di primo sorpasso seconda formula}
\end{equation}%
We observe that, in both (\ref{Tempi di primo sorpasso prima formula}) and (%
\ref{Tempi di primo sorpasso seconda formula}), we have that
\begin{equation*}
Pr(\mathcal{T}_{1}>t)=e^{-\int_{0}^{t}\lambda ^{\alpha (\tau
)}d\tau}=Pr(T_{1}>t),
\end{equation*}%
because the time when the first jump occurs (i.e. $T_{1}$) obviously
coincides with the surpassing time of the level $k=1$ (i.e. $\mathcal{T}_{1}$%
).

Here is the proof of (\ref{Tempi di primo sorpasso seconda
formula}), in the non-trivial case $k\geq 2$:
\begin{align}
Pr(\mathcal{T}_{k}>t)& =Pr(X(t)<k)  \notag \\
& =Pr(X(t)=0)+\sum_{n=1}^{k-1}Pr(X(t)=n)  \notag \\
& =e^{-\int_{0}^{t}\lambda ^{\alpha
(s)}ds}+\sum_{n=1}^{k-1}\sum_{r=1}^{\infty }\frac{(-1)^{n+r}}{r!}%
\int_{[0,t]^{r}}\lambda ^{\beta _{r}(s)}\binom{\beta _{r}(s)}{n}%
ds_{1}...ds_{r}  \notag \\
& =e^{-\int_{0}^{t}\lambda ^{\alpha (s)}ds}+\sum_{r=1}^{\infty }\frac{%
(-1)^{r}}{r!}\int_{[0,t]^{r}}\lambda ^{\beta _{r}(s)}\biggl (%
\sum_{n=1}^{k-1}(-1)^{n}\binom{\beta _{r}(s)}{n}\biggr
)ds_{1}...ds_{r},\label{upcrossing}
\end{align}%
where we used  (\ref{probabilita di
stato}) putting $\tau =0$. By using the following relation \footnote{Such a formula
can be proved for $k=1$ and then generalized to $k>1$, by a
standard use of the principle of induction.}
\begin{equation*}
\sum_{n=0}^{k-1}(-1)^{n}\binom{x}{n}=(-1)^{k+1}\frac{k}{x}\binom{x}{k},
\end{equation*}%
formula (\ref{upcrossing}) reduces to
\begin{equation*}
1-k\sum_{r=1}^{\infty }\frac{(-1)^{r+k}}{r!}\int_{[0,t]^{r}}\frac{\lambda
^{\beta _{r}(s)}}{\beta _{r}(s)}\binom{\beta _{r}(s)}{k}ds_{1}...ds_{r}
\end{equation*}%
and, by writing
\begin{equation*}
\frac{\lambda ^{\beta _{r}(s)}}{\beta _{r}(s)}=\int_{0}^{\lambda }(\lambda
^{\prime })^{\beta _{r}(s)-1}d\lambda ^{\prime },
\end{equation*}%
equation (\ref{Tempi di primo sorpasso seconda formula}) is
immediately obtained.

\section{Time-Multifractional Poisson Process}

\subsection{Inverse multistable process}

Let $H$ be a multistable subordinator. Since $H$ is a cadlag process, with
strictly increasing trajectories, and such that $H(0)=0$ and $H(\infty
)=\infty $ almost surely, then the hitting-time process
\begin{equation}
L(x)=inf\{t \geq 0:H(t)>x\}\label{definizione inverso}
\end{equation}%
is well defined and has continuous sample paths. Together with (\ref%
{definizione inverso}), the following definition holds
\begin{equation*}
H(x^{-})=sup\{t \geq 0:L(t)<x\}.
\end{equation*}
In the time-homogeneous case, it is well known that, if $H$ is
stable with index $\alpha $, both $H$ and its inverse $L$ are
self-similar, with exponent $1/\alpha $ and $\alpha $
respectively, that is the following relations hold in distribution
(see, for example, \cite{Meerschaert2}):
\begin{equation*}
H(ct)=c^{\frac{1}{\alpha }}H(t),\qquad \qquad L(ct)=c^{\alpha }L(t).
\end{equation*}%
In the non-homogeneous case, the process $H$ is not self-similar, but its
local approximation has this property. More precisely, the multistable
subordinator is localizable (see, for example, \cite{legueve2} and \cite{Orsingher}), in the
sense that
\begin{equation}\label{localizzabilita}
\lim_{r\rightarrow 0}\frac{H(t+rT)-H(t)}{r^{\frac{1}{\alpha (t)}}}   \stackrel{\textrm{law}}{=} Z_{t}(T),
\end{equation}%
where $Z_{t}$ is the local (or tangent) process at $t$ and consists of a
homogeneous stable process with index $\alpha (t)$. We now investigate the
behaviour of the inverse process.

\begin{prop}
The process $L$ defined in (\ref{definizione inverso}) is localizable, and
the tangent process is given by the inverse $\mathcal{L}_t$ of $Z_{t}$.
\end{prop}

\begin{proof}
By (\ref{localizzabilita}), for each $t\geq 0$, we can write
\begin{equation}
\lim_{r\rightarrow 0}Pr\biggl (\frac{H(t+rT)-H(t)}{r^{\frac{1}{\alpha (t)}}}%
\leq w\biggr )=Pr(Z_{t}(T)\leq w)= \Pr \{\mathcal{L}_t(w)\geq T  \}.\label{localizzabilità multistabile}
\end{equation}%
Since $H$ has independent increments, $H(t+rT)-H(t)$ is independent of $H(t)$%
. So we can condition on $H(t)=x$, without changing the left-hand
side of (\ref{localizzabilità multistabile}), which can be
written as
\begin{align}
&\lim_{r\rightarrow 0}\Pr \biggl (H(t+rT)-H(t) \leq wr^{\frac{1}{\alpha (t)}}%
 |H(t)=x \biggr )  \label{passaggi} \\
& =\lim_{r\rightarrow 0}\Pr \biggl (L(x+wr^{\frac{1}{\alpha (t)}})-L(x)\geq
rT\biggr )\notag \\
& =\lim_{r^{\prime }\rightarrow 0}\Pr \biggl (\frac{L(x+wr^{\prime })-L(x)}{%
r^{\prime }{}^{\alpha (t)}}\geq T\biggr ),\notag
\end{align}%
where, in the last step, we made the substitution $r^{\prime }=r^{\frac{1}{%
\alpha (t)}}$. Thus
\begin{align*}
 \lim_{r^{\prime }\rightarrow 0}\Pr \biggl (\frac{L(x+wr^{\prime })-L(x)}{%
r^{\prime }{}^{\alpha (t)}}\geq T\biggr )   = \Pr \{\mathcal{L}_t(w)\geq T \}
\end{align*}
and the proof is complete.
\end{proof}

\subsection{Paths and distributional properties}
Let $N$ be a Poisson process with intensity $\lambda >0$, and let $L$ be the
inverse of a multistable subordinator independent of $N$.
We define the TMPP as the time-changed process $N(L(t))$.
Since $N$ is one-stepped and $L$ is continuous, then $N(L(t))$ is also a one-stepped continuous time random walk defined as
\begin{align*}
N(L(t))=n \qquad \Longleftrightarrow \qquad T_n<t<T_{n+1} \qquad n=0,1,2...
\end{align*}
where $T_0=0$ a.s. and, for $n\geq 1$, $T_n=J_1+...+J_n$, $J_n$ being the waiting time for the state $n$.
The construction of the process is
contained in the following result.

\begin{te}
The time-changed Poisson process $\left\{ N(L(t)),t \geq 0\right\} $ is a
one-stepped counting process with independent waiting times $J_n, n\geq 1$,
each having Laplace transform
\begin{align}
\mathbb{E}e^{-\eta J_{n}}& =\int \int_{0<u<v<\infty }\frac{\lambda
^{n}e^{-\lambda v}u^{n-2}}{\Gamma (n-1)}e^{-\int_{u}^{v}\eta ^{\alpha (\tau
)}d\tau }\,dudv\qquad n\geq 2  \label{MittagLeffler generalizzata} \\
\mathbb{E}^{-\eta J_{1}}& =\int_{0}^{\infty }\,dw\lambda e^{-\lambda
w}e^{-\int_{0}^{w}\eta ^{\alpha (\tau ) }d\tau}.
\label{MittagLeffler generalizzata primo intertempo}
\end{align}
\end{te}

\begin{proof}
Let $W_{n}$, $n\geq 1$ be the i.i.d waiting times between jumps of a Poisson
process, so that $\Pr (W_{n}\in dw)=\lambda e^{-\lambda w}dw$. Let $%
V_{n}=W_{1}+W_{2}+...+W_{n}$, $n\geq 1$, be the hitting times of $N(t)$,
each having distribution $\Pr (V_{n}\in du)=\Gamma (n)^{-1}\lambda
^{n}e^{-\lambda u}u^{n-1}du, \,u>0$.

The joint distribution of two successive hitting times reads
\begin{align*}
\Pr (V_{n-1}\in du,V_{n}\in dv)& =\Pr (V_{n-1}\in du,W_{n}\in d(v-u)) \\
& =\Pr (V_{n-1}\in du)\Pr (W_{n}\in d(v-u)) \\
& =\frac{\lambda ^{n}e^{-\lambda v}u^{n-2}}{\Gamma (n-1)}dudv\qquad
0<u<v<\infty.
\end{align*}
Now let $T _{1}....T _{n}$ be the hitting times of $N(L(t))$, such
that
\begin{equation*}
T _{n}=sup\{t \geq 0:L(t)<V_{n}\}.
\end{equation*}%
Since $L$ is the right continuous inverse of $H$, it follows that $T
_{n}=H(V_{n}^{-})$ and this, together with the fact that $H$ is a.s.
continuous for any $t \geq 0$ (see \cite{Orsingher}), implies that $T_n=H(V_{n})$ in distribution.
The waiting times between jumps of $N(L(t))$ are defined as $J_{n}=T _{n}-T _{n-1}$, where $n\geq 1$. For
 $n=1$ we have that
\begin{equation*}
\mathbb{E}e^{-\eta J_{1}}=\mathbb{E}e^{-\eta H(W_{1})}=\mathbb{E}\bigl[
\mathbb{E}\bigl(e^{-\eta H(W_{1})}|W_1\bigr)\bigr]=\int_{0}^{\infty }dw\lambda e^{-\lambda
w}e^{-\int_{0}^{w}\eta ^{\alpha (\tau )}d\tau }
\end{equation*}%
while, for $n\geq 2$
\begin{align*}
\mathbb{E}e^{-\eta J_{n}}& =\mathbb{E}e^{-\eta (T _{n}-T _{n-1})} \\
& =\mathbb{E}e^{-\eta (H(V_{n})-H(V_{n-1}))} \\
& =\mathbb{E}  [\mathbb{E}(e^{-\eta
(H(V_{n})-H(V_{n-1})}|V_{n-1},V_{n})] \\
& =\int \int_{0<u<v<\infty }\mathbb{E}e^{-\eta (H(v)-H(u))}P(V_{n}\in
dv,V_{n-1}\in du) \\
& =\int \int_{0<u<v<\infty }e^{-\int_{u}^{v}\eta ^{\alpha (\tau ) }d\tau}\,
\frac{\lambda ^{n}e^{-\lambda v}u^{n-2}}{\Gamma (n-1)}dudv,
\end{align*}%
and the proof is complete.
\end{proof}

\begin{os}
In the time-homogeneous case, where $\alpha (s)=\alpha \in (0,1)$, the
random time $L$ reduces to the inverse stable subordinator with index $%
\alpha $. A slight calculation shows that expressions (\ref{MittagLeffler
generalizzata}) and (\ref{MittagLeffler generalizzata primo intertempo})
become independent of the state $n$ and have the following form
\begin{equation*}
\mathbb{E}e^{-\eta J_{n}}=\frac{\lambda }{\eta ^{\alpha }+\lambda },\qquad
\forall n\geq 1.
\end{equation*}%
Thus we have a renewal process with i.i.d waiting times which satisfy
\begin{equation*}
Pr(J_{n}>t)=E_{\alpha }(-\lambda t^{\alpha }),
\end{equation*}%
where
\begin{equation*}
E_{\alpha }(z)=\sum_{k=0}^{\infty }\frac{z^{k}}{\Gamma (1+\alpha k)}
\end{equation*}%
is the Mittag-Leffler function. So, in the homogeneous case, $N(L(t))$ reduces to the celebrated
time-fractional Poisson process, which is a renewal process with
Mittag-Leffler waiting times (see, for example, \cite{Beghin2}). In such a case, the one-dimensional state probabilities $p_{k}(t)=\Pr
(N(L(t))=k)$ solve the following system of fractional difference-differential equations
\begin{equation}
\begin{cases}
\label{equazione frazionaria nel tempo classica}\frac{\partial ^{\alpha }}{%
\partial t^{\alpha }}p_{k}(t)=-\lambda p_{k}(t)+\lambda p_{k-1}(t)\qquad
k\geq 1 \\
\frac{\partial ^{\alpha }}{\partial t^{\alpha }}p_{0}(t)=-\lambda p_{0}(t)
\\
p_{k}(0)=\delta _{k,0},%
\end{cases}%
\end{equation}%
where
\begin{align*}
\frac{\partial ^{\alpha }}{%
\partial t^{\alpha }}f(t)= \frac{1}{\Gamma (1-\alpha)} \int _0^t (t-s) ^{-\alpha} f'(s)ds
\end{align*}
is the Caputo derivative of order $\alpha \in (0,1)$.
As shown in \cite{Beghin}, the solution to (\ref{equazione frazionaria nel
tempo classica}) is such that
\begin{equation}
\tilde{p}_{k}(s)=\int_{0}^{\infty }e^{-st}p_{k}(t)dt=\frac{\lambda
^{k}s^{\alpha -1}}{(s^{\alpha }+\lambda )^{k+1}}\qquad k\geq 0
\label{soluzione time fractional poisson classico}
\end{equation}%
\end{os}

We now find the multifractional analogue of formula (\ref{soluzione time fractional
poisson classico}).

\begin{prop}
Let $N(L(t))$ be a TMPP and let $%
p_{k}(t)=\Pr \{N(L(t))=k\}$ be its state probabilities. Then
\begin{align}\label{legge di stato iniziale generalizzata}
\tilde{p}_0 (s)&= \frac{1}{s}-\frac{1}{s} \int _0 ^{\infty}dw\, \lambda e ^{-\lambda w} e^{-\int _0^w s^{\alpha (\tau)}d\tau},\\
\label{legge di stato generalizzata}
\tilde{p}_{k}(s)&=\int_{0}^{\infty }dx e^{-\int_{0}^{x}s ^{\alpha (\tau
) }d\tau }\frac{\lambda ^{k}x^{k}e^{-\lambda x}(kx^{-1}-\lambda )}{s\Gamma
(k+1)}, \qquad k\geq 1.
\end{align}
\end{prop}

\begin{proof} Consider that $p_0(t)= P(J_1>t)$.  By deriving this relation with respect to $t$ and taking the Laplace transform,  (\ref{MittagLeffler generalizzata primo intertempo}) leads to (\ref{legge di stato iniziale generalizzata}).

Let now $T _{k}, k\geq 1$ be the hitting times of $N(L(t))$. As explained
in the proof of the previous theorem, $T_k=H(V_k)$ in distribution, that is
\begin{equation}
\mathbb{E}e^{-sT _{k}}=\int_{0}^{\infty }dx\frac{\lambda
^{k}x^{k-1}e^{-\lambda x}}{\Gamma (k)}e^{-\int_{0}^{x}s ^{\alpha (\tau
) }d\tau}.\label{aaa}
\end{equation}%
By considering that
\begin{equation*}
\Pr \{N(L(t))\geq k\}=\Pr \{T _{k}<t\}, \qquad k\geq 1,
\end{equation*}%
we have that $\Pr \{N(L(t))=k\}=\Pr \{T _{k}<t\}-\Pr \{T _{k+1}<t\}$.
By deriving with respect to $t$ and taking the Laplace transform, we obtain
\begin{equation*}
s\tilde{p}_{k}(s)=\mathbb{E}e^{-sT _{k}}-\mathbb{E}e^{-sT _{k+1}}
\end{equation*}%
and, using (\ref{aaa}), formula (\ref{legge di stato generalizzata}) follows.
\end{proof}

It is straightforward to note that  (\ref{legge di stato iniziale generalizzata})
and (\ref{legge di stato generalizzata}) reduce to (\ref{soluzione time
fractional poisson classico}) by assuming $\alpha (t)$ to be constant with
respect to $t$.

\subsection{Concluding remarks}
Unfortunately, in the time-inhomogeneous case, the connection with fractional
calculus is not immediate as in the classical case. Indeed, in \cite{Orsingher}
the authors proposed  an equation governing Markovian processes time-changed
via the inverses of inhomogeneous subordinators. Such equation involves
generalized fractional derivatives, but it is not easy to handle,
especially because it does not involve the distribution of the time-changed process
only, but also the distributions of both the original Markov process and the operational time.

We finally observe that our construction of the TMPP extends to the inverses of arbitrary
non-homogeneous subordinators provided that $\nu _{t}(\mathbb{R}^{+})=\infty $.
Indeed, let $N$ be an ordinary Poisson process and let $L$ be the inverse of
 any non-homogeneous subordinator. Then $N(L(t))$ is a counting process with
independent intertimes given by
\begin{align}
\mathbb{E}e^{-\eta J_{n}}& =\int \int_{0<u<v<\infty }\frac{\lambda
^{n}e^{-\lambda v}u^{n-2}}{\Gamma (n-1)}e^{-\int_{u}^{v}f(\eta ,\tau )d\tau
}\,dudv\qquad n\geq 2   \\
\mathbb{E}^{-\eta J_{1}}& =\int_{0}^{\infty }\,dw\lambda e^{-\lambda
w}e^{-\int_{0}^{w}f(\eta ,\tau )d\tau }
\end{align}
To prove this, it is sufficient to adapt the same construction given in the
proof of Theorem 4.2, using the inverse process of an inhomogeneous
subordinator with Bernstein function of the form $f(x,t)$.
Of course, in the homogeneous case, where $f$ is independent of $t$,
we obtain a renewal process with i.i.d intertimes such that
\begin{align*}
\mathbb{E}e^{-\eta J_n}= \frac{\lambda }{\lambda +f(s)}, \qquad n\geq 1,
\end{align*}
which has been analysed in \cite{Meerschaert1}.

\end{document}